\documentclass[12pt]{amsart}
\usepackage{epsfig,color}
\usepackage{blindtext}
\usepackage{hyperref}
\usepackage{graphicx}
\usepackage{enumitem} 
\usepackage{url}
\usepackage{amssymb}
\usepackage{graphicx,import}
\usepackage{comment}

\theoremstyle{definition}

\newtheorem{theorem}{Theorem}[section]
\newtheorem{definition}[theorem]{Definition}

\newtheorem{proposition}[theorem]{Proposition}
\newtheorem{lemma}[theorem]{Lemma}
\newtheorem{remark}[theorem]{Remark}
\newtheorem{corollary}[theorem]{Corollary}

\newtheorem{question}[theorem]{Question}

\newtheorem*{remark*}{Remark}

\numberwithin{equation}{section}

\newcommand{\cN}{\mathcal{N}}
\newcommand{\cH}{\mathcal{H}}

\newcommand{\cB}{\mathcal{B}}
\newcommand{\cS}{\mathcal{S}}

\newcommand{\bH}{\mathbf{H}}

\newcommand{\R}{\mathbb{R}}

\DeclareMathOperator{\vol}{Vol}
\DeclareMathOperator{\area}{Area}

\DeclareMathOperator{\spt}{spt}
\DeclareMathOperator{\Graph}{Graph}

\DeclareMathOperator{\diam}{diam}
\DeclareMathOperator{\loc}{loc}

\title[Entropy and Partial Regularity]{Entropy in A Closed Manifold and Partial Regularity of Mean Curvature Flow Limit of Surfaces}

\date{\today}

\author{Ao Sun}
\address{Department of Mathematics, Massachusetts Institute of Technology, Cambridge, MA 02139, USA}
\email{aosun@mit.edu}

\begin{document}

\begin{abstract}
Inspired by the idea of Colding-Minicozzi in \cite{CM1}, we define (mean curvature flow) entropy for submanifolds in a general ambient Riemannian manifold. In particular, this entropy is equivalent to area growth of a closed submanifold in a closed ambient manifold with non-negative Ricci curvature. Moreover, this entropy is monotone along the mean curvature flow in a closed Riemannian manifold with non-negative sectional curvatures and parallel Ricci curvature. As an application, we show the partial regularity of the limit of mean curvature flow of surfaces in a three dimensional Riemannian manifold with non-negative sectional curvatures and parallel Ricci curvature.
\end{abstract}

\maketitle

\section{Introduction}
The mean curvature flow entropy, firstly introduced by Colding-Minicozzi \cite{CM1} to study mean curvature flow singularities in $\R^n$, is a very important quantity characterizing all the scales of a submanifold in $\R^n$. In this paper, we follow the idea of Colding-Minicozzi to define the entropy of submanifolds in a general ambient Riemannian manifold. 

Let us first recall the entropy defined by Colding-Minicozzi. For a $m$-dimensional submanifold $M\subset\R^n$, the entropy $\lambda$ of $M$ is defined by
\begin{equation}
\lambda(M)=\sup_{x_0\in\R^n,t_0\in\R^+}\frac{1}{(4\pi t_0)^{m/2}}\int_{M}e^{-\frac{|x-x_0|^2}{4t_0}}d\mu_{M}.
\end{equation}
The entropy is monotone along the mean curvature flow of submanifolds by Huisken's monotonicity formula, see \cite{Hu} and \cite{CM1}. Therefore the entropy is an important quantity in the study of mean curvature flow. It is known that 
\[\rho_{x_0,t_0}(x)=\frac{1}{(4\pi t_0)^{n/2}}e^{-\frac{|x-x_0|^2}{4t_0}}\]
is just the heat kernel at a fixed time $t_0$ and a fixed point $x_0$ in the Euclidean space $\R^n$. The entropy is the supremum among all the possible integrals of the heat kernel, with a dimensional correction constant.

Inspired by Colding-Minicozzi's definition, in a general $n$-dimensional ambient Riemannian manifold $\cN$, we define the entropy of a $m$-dimensional submanifold $M$. Throughout this paper, a closed Riemannian manifold is a compact Riemannian manifold without boundary.
\begin{definition}
Let $\cN$ be a $n$-dimensional closed or complete Riemannian manifold. Let $M$ be a $m$-dimensional submanifold of $\cN$. Then we define the {\bf entropy} $\lambda$ of $M$ to be
\begin{equation}
\lambda(M)=\sup_{x\in \cN,t\in\R^+}t^{(n-m)/2}\int_M \cH(x,y,t)dy.
\end{equation}
Here $\cH$ is the heat kernel on the manifold $\cN$.
\end{definition}

To the author's knowledge, this is first intrinsic generalization of Colding-Minicozzi's entropy to a curved ambient space. There are some extrinsic generalizations of Colding-Minicozzi's entropy have been studied before, see \cite{M} and \cite{Z}. These definitions are extrinsic because they are defined by isometrically embedding the ambient manifold into $\R^N$, and then using Colding-Minicozzi's theory in $\R^N$. These quantities are not monotone decreasing along a mean curvature flow. We also remark that very recently Bernstein \cite{B} defined a version of Colding-Minicozzi's entropy in a hyperbolic space.

\bigskip
The entropy will be used to study mean curvature flow in a closed Riemannian manifold. In particular, we highlight the following theorem of the partial regularity of the limit of long-time mean curvature flow. Recall that a mean curvature flow $M_t$ is called a {\bf long-time} mean curvature flow if it is defined for $t\in[0,\infty)$.

\begin{theorem}\label{thm:Main theorem 2}
Let $\cN$ be a closed $3$-dimensional Riemannian manifold with non-negative sectional curvatures and parallel Ricci curvature. Suppose $M_t$ is a long-time mean curvature flow of closed embedded surfaces in $\cN$, then there exists sequence $t_i\to\infty$ such that $M_{t_i}$ converges to a varifold $V$ in the sense of varifolds, and $\spt V$ is a smoothly embedded minimal surface.
\end{theorem}

For the definitions of varifolds and the support of a varifold ($\spt V$), we refer the readers to the standard context in geometric measure theory, for example \cite{Si}.

\begin{remark}
There are many closed manifolds $\cN$ which satisfy the curvature condition in the statement of Theorem \ref{thm:Main theorem 2}, for example the sphere $S^3$ and the torus $T^3$ with standard metrics. Therefore suppose $M_t$ is a long-time mean curvature flow of closed embedded surfaces in $S^3$ or $T^3$ with standard metrics, then there exists a sequence $t_i\to\infty$ such that $M_{t_i}$ converges to a varifold $V$ in the sense of varifolds, and $\spt V$ is a smoothly embedded minimal surface.
\end{remark}

Recall that a mean curvature flow $M_t$ is {\bf ancient} if it is defined for $t\in(-\infty,0)$. If we assume the entropy bound of an ancient mean curvature flow is finite, we obtain the following partial regularity of the backward limit of the flow.

\begin{theorem}\label{thm:Main theorem 3}
Let $\cN$ be a closed $3$-dimensional Riemannian manifold with non-negative sectional curvatures and parallel Ricci curvature. Suppose $M_t$ is an ancient mean curvature flow of closed embedded surfaces in $\cN$, and $\lambda(M_t)<\infty$. Then there exists sequence $t_i\to-\infty$ such that $M_{t_i}$ converges to a varifold $V$ in the sense of varifolds, and $\spt V$ is a smoothly embedded minimal surface.
\end{theorem}

Since mean curvature flow is the gradient flow of area functional, it is believed that there would be a Morse theory based on the area functional that can be studied by mean curvature flow. Nevertheless, the singular behaviour of mean curvature flow is very complicated, hence this approach is far from well understood. Our main theorems show that if there is no singularity appearing during the mean curvature flow surfaces in a closed Riemannian manifold with certain curvature assumptions, and then such a flow really connects two critical points (i.e. minimal surfaces). This might serve as a step towards the understanding of mean curvature flow in a closed Riemannian manifold.

The partial regularity follows Ilmanen's idea in \cite{I}. \cite{I}, Ilmanen proved the partial regularity of the tangent flow of mean curvature flow of closed embedded surfaces in $\R^3$. For the sake of completeness, we discuss the proof of Ilmanen and how to use it to prove our main theorem in Section \ref{S:Partial Regularity}.

\subsection{Monotonicity}
The monotonicity of entropy in a closed Riemannian manifold relies on a monotonicity formula proved by Hamilton in \cite{H2}. Unlike the monotonicity formula of mean curvature flow in $\R^n$ proved by Huisken in \cite{Hu}, Hamilton's monotonicity requires the ambient manifold having some curvature assumptions. Let $\cN$ be a closed Riemannian manifold and $M_t$ be a mean curvature flow in $\cN$. The monotonicity formula requires that $\cN$ has non-negative sectional curvatures and parallel Ricci curvature. In fact, in the computations to obtain the monotonicity formula, one term has to be non-negative, and it is non-negative if a Harnack estimate holds. This Harnack estimate has been studied by Hamilton in \cite{H1}, and the curvature assumption is necessary. 

In conclusion, we obtain the monotonicity of entropy along the mean curvature flow in a special family of closed manifolds.

\begin{theorem}
Suppose $\cN$ has non-negative sectional curvatures and parallel Ricci curvatures. Let $M_t$ be a mean curvature flow in $\cN$. Then the entropy $\lambda(M_t)$ is monotone non-increasing along the mean curvature flow.
\end{theorem}

In a general Riemannian manifold, we still have an almost monotonicity formula of entropy. However when $t\to\infty$, the entropy becomes unbounded, and we cannot use it to study the limit behaviour of mean curvature flows. All the details are discussed in Section \ref{S:Monotonicity Formula}.

\subsection{Area growth}
In this paper, we will use ``area" to denote the $m$-dimensional Hausdorff measure of a $m$-dimensional submanifold, and ``volume" to denote the $n$-dimensional Hausdorff measure of the $n$-dimensional ambient Riemannian manifold. The following property is essential when we use entropy to study mean curvature flow in a closed Riemannian manifold.

\begin{theorem}\label{Thm:entropy bound area growth}
Let $\cN$ be a closed Riemannian manifold with non-negative Ricci curvature. There exists a constant $C$ only depending on the geometry of $\cN$ such that the following inequality holds: for any $m$-dimensional submanifold $\Sigma\subset\cN$,
\begin{equation}
C^{-1}\sup_{x\in\cN,r>0}\frac{\area(\Sigma\cap B_r(x))}{r^m} \leq \lambda(\Sigma)\leq C \sup_{x\in\cN,r>0}\frac{\area(\Sigma\cap B_r(x))}{r^m}.
\end{equation}
Here $B_r(x)$ is the geodesic ball in $\cN$.
\end{theorem}
We call $\sup_{x\in\cN,r>0}\frac{\area(\Sigma\cap B_r(x))}{r^m}$ the area growth of a submanifold $\Sigma$. In other words, the area growth and the entropy of a submanifold bound each other in a Ricci non-negative closed Riemannian manifold. 

In the Euclidean space $\R^n$, this fact has been realized by Colding-Minicozzi in \cite{CM1}. In $\R^n$, the Gaussian density function can be viewed as a regularization of the characteristic function of a ball, so the integral of the Gaussian density is very similar to the integral of a characteristic function of a ball. This is the intuition about the equivalence of the entropy and the area growth.

There are many studies about the heat kernel in a closed Riemannian manifold. In \cite{CY}, Cheeger-Yau studied the lower bound of the heat kernel on a closed Riemannian manifold. Later in an influential paper \cite{LY} by Li-Yau, a very precise upper bound and lower bound of heat kernel on a closed Riemannian manifold with non-negative Ricci curvature have were developed. The heat kernel on a closed Riemannian manifold with non-negative Ricci curvature is bounded by some Gaussian distribution from above and below. Therefore, the entropy on a closed Riemannian manifold with non-negative Ricci curvature is also equivalent to the area growth. We show this in Section \ref{S:Area Growth}.

We want to mention that the relation between Perelman's entropy and the volume growth of Ricci flow has been studied in \cite{Pe} by Perelman and in \cite{Ni} by Ni. 

\subsection*{Acknowledgement}
The author wants to thank Professor Bill Minicozzi for his helpful comments. The author is also grateful to Zhichao Wang, Jinxin Xue and Xin Zhou for the invaluable discussions. Finally, the author thanks the anonymous referees for the comments and suggestions.

\section{Entropy}
Let us first recall some basic properties of the heat kernel in a closed or complete Riemannian manifold.

\begin{theorem}[\cite{SY}, Section 3 Theorem 2.1]
Let $\cN$ be a closed or complete Riemannian manifold. There exists a heat kernel $\cH(x,y,t)\in C^\infty(\cN\times\cN\times \R^+)$ satisfying the following properties:
\begin{enumerate}
\item $\cH(x,y,t)=\cH(y,x,t)$,
\item $\lim_{t\to\ 0^+}\cH(x,y,t)=\delta_x(y)$,
\item $\partial_t\cH=\Delta_y \cH$,
\item $\cH(x,y,t)=\int \cH(x,z,t-s)\cH(z,y,s)dz$.
\end{enumerate}
\end{theorem}

We define the backward heat kernel
\[\rho_{y,T}(x,t)=\cH(x,y,T-t).\]
Then $\rho_{y,T}$ satisfies the backward heat equation for $t\in(-\infty,T)$. We follow Colding-Minicozzi's idea in \cite{CM1} to define $F$-functional and entropy.

\begin{definition}
Let $\cN$ be a $n$-dimensional Riemannian manifold and $\Sigma\subset\cN$ be a $m$-dimensional submanifold. Given $(x,t)\in \cN\times \R^+ $, we define
\begin{equation}
F_{x,t}(\Sigma)=t^{(n-m)/2}\int_\Sigma \rho_{x,t}(y,0)dy=t^{(n-m)/2}\int_\Sigma \cH(x,y,t)dy.
\end{equation}

We define the {\bf entropy} $\lambda$ of $\Sigma$ to be 
\begin{equation}
\lambda(\Sigma)=\sup_{x\in M,t\in\R^+}F_{x,t}(\Sigma).
\end{equation}
\end{definition}

\bigskip
Next let us discuss some basic properties of the entropy. If the ambient space is the Euclidean space $\R^n$, these properties have been studied by Colding-Minicozzi in \cite{CM1}.

We will use $\lambda^\cN$ to denote the entropy of a submanifold in $\cN$ if we want to specify the ambient space $\cN$.
 
\begin{proposition}[Product property]
Let $\cN$ be a $n$-dimensional Riemannian manifold and $\cN'$ be a $n'$-dimensional Riemannian manifold. Let $M\subset\cN$ be a $m$-dimensional submanifold and $M'\subset\cN'$ be a $m'$-dimensional submanifold. Then
\begin{equation}
\lambda^{\cN\times\cN'}(M\times M')=\lambda^{\cN}(M)\lambda^{\cN'}(M').
\end{equation}
\end{proposition}

\begin{proof}
This is an easy consequence of the fact that on a product Riemannian manifold, the heat kernel satisfies the product property: suppose $x,y\in\cN$, $x',y'\in\cN'$, then
\begin{equation}
\cH^{\cN\times\cN'}((x,x'),(y,y'),t)=\cH^{\cN}(x,y,t)\cH^{\cN'}(x',y',t).
\end{equation} 
\end{proof}

\begin{proposition}[Group action invariance]
Suppose $G$ is an isometric group action on $\cN$, then
\begin{equation}
\lambda(M)=\lambda(G(M)).
\end{equation}
\end{proposition}

\begin{proof}
This is an easy consequence of the fact that the entropy is the supremum among all $F$-functionals, and $G$ induces a natural group action on the space of all $F$-functionals.
\end{proof}

\section{Monotonicity Formula}\label{S:Monotonicity Formula}
We first recall the monotonicity formula of mean curvature flow in a closed Riemannian manifold by Hamilton \cite{H2}. We say a function $k$ is a backward solution to the heat equation on a Riemannian manifold $\cN$ if it solves the equation
\[\partial_t k=-\Delta k.\]

\begin{theorem}[\cite{H2}, Theorem B]\label{Thm:Hamilton's monotonicity}
Let $M_t$ be a $m$-dimensional mean curvature flow in a $n$-dimensional closed Riemannian manifold $\cN$, defined on $0\leq t<T$. If $k$ is any positive backward solution to the scalar heat equation on $\cN$ with $\int_{\cN} k =1$, then the quantity
\[Z(t)=(T-t)^{(m-n)/2}\int_{M_t}k d\mu_t\]
is monotone decreasing in $t$ when $\cN$ is Ricci parallel with non-negative sectional curvatures. On a general $\cN$ we have
\[Z(t)\leq C Z(s)+C(t-s)A_0\]
whenever $T-1\leq s\leq t\leq T$, where $A_0$ is the initial area at time $0$ and $C$ is a constant depending only on the geometry of $\cN$.
\end{theorem}

\begin{remark}
The curvature condition is used in the computation of the monotonicity formula. Actually we have
\begin{equation}
\frac{d}{dt}Z(t)= -(T-t)^{(m-n)/2}\int_{M_t}\left|\vec H -\frac{Dk}{k}\right|^2 k d\mu_t-(T-t)^{(m-n)/2}\int_{M_t}Qd\mu_t,
\end{equation}
where $Q$ is the Harnack form

\begin{equation}
Q=g^{\alpha\beta}\left(D_\alpha D_\beta l-\frac{D_\alpha k D_\beta k}{k}+\frac{1}{(T-t)}kg_{\alpha\beta}\right).
\end{equation}

Here we use indices $\alpha$, $\beta$, $\cdots$ for local coordinates on $M_t$ to write $Q$ as Hamilton did in \cite{H2}. Hamilton proved that $Q\geq 0$ if $\cN$ has non-negative sectional curvatures and parallel Ricci curvatures, see \cite{H1}. The general case follows from a weak version of the Harnack inequality.
\end{remark}

Then by adapting Hamilton's monotonicity formula, we obtain the following monotonicity of $F$-functional and entropy.
\begin{theorem}\label{Thm:entropy monotonicity}
Suppose $\cN$ has non-negative sectional curvatures and parallel Ricci curvatures. Let $M_t$ be a mean curvature flow in $\cN$. Then for any $x\in \cN$, $t_1\leq t_2$, we have
\begin{equation}
F_{x,s}(M_{t_2})\leq F_{x,s+(t_2-t_1)}(M_{t_1})
\end{equation}
if the terms in the inequality are well-defined. Taking supremum among $x\in\cN$ and $s\in\R^+$ gives the monotonicity of entropy

\begin{equation}
\lambda(M_{t_2})\leq \lambda(M_{t_1}).
\end{equation}
\end{theorem}

Note that if the ambient Riemannian manifold is $\R^{n}$, then the above monotonicity formula is just the famous monotonicity formula by Huisken, see \cite{Hu} and \cite{CM1}.

We also have the following monotonicity formula in a general closed Riemannian manifold.

\begin{theorem}
Let $\cN$ be a closed Riemannian manifold, and let $M_t$ be a mean curvature flow in $\cN$. Then for any $x\in\cN$, $t_2-1 \leq t_1\leq t_2$, we have
\begin{equation}
F_{x,s}(M_{t_2})\leq C F_{x,s+(t_2-t_1)}(M_{t_1}) +C(t_2-t_1)A_0,
\end{equation}
if the terms in the inequality are well-defined. Taking supremum among $x\in\cN$ and $s\in\R^+$ gives 
\begin{equation}
\lambda(M_{t_2})\leq C\lambda(M_{t_1})+C(t_2-t_1)A_0.
\end{equation}
\end{theorem}

So far we obtain some nice monotonicity formula for the entropy. However, if the entropy of a closed embedded submanifold is infinite, then the monotonicity is meaningless. We will prove in the next section that the entropy of a closed embedded submanifold in a closed Ricci non-negative Riemannian manifold is always finite. See Proposition \ref{Prop:entropy is finite}.
 
\section{Area Growth}\label{S:Area Growth}
The goal of this section is to prove Theorem \ref{Thm:entropy bound area growth}, which says that the entropy can bound the area growth of a submanifold. In particular, together with monotonicity of entropy, we obtain a uniform area growth bound for mean curvature flow. Throughout this section, the constant $C$ varies line to line, but only depends on the geometry of $\cN$.

We use the notation that $B_r(x)$ is a geodesic ball in $\cN$

\begin{definition}
The {\bf area growth bound} $\kappa$ of a $m$-dimensional submanifold $\Sigma$ is 
\[\kappa(\Sigma)= \sup_{x\in \cN,r>0}\frac{\area(\Sigma\cap B_r(x))}{r^m}\]
\end{definition}

Before we prove Theorem \ref{Thm:entropy bound area growth}, let us recall the lower bound of the heat kernel on a Riemannian manifold with non-negative Ricci curvature. We use $V_x(r)$ to denote the volume of the geodesic ball $B_r(x)$, and we use $r_d$ to denote the distance function on $\cN$.

\begin{theorem}[\cite{LY}, Section 3 \& Section 4]\label{Thm:Heat kernel bound}
Let $\cN$ be a closed Riemannian manifold with non-negative Ricci curvature. The heat kernel $\cH$ on $\cN$ satisfies the following inequalities
\begin{equation}
\cH(x,y,t)\geq C V_x^{-\frac{1}{2}}(\sqrt{t})V_y^{-\frac{1}{2}}(\sqrt{t})e^{-\frac{r_d^2(x,y)}{3t}},
\end{equation}

and for any $\epsilon\in(0,1)$,
\begin{equation}
\cH(x,y,t)\leq C(\epsilon)V_x^{-\frac{1}{2}}(\sqrt{t})V_y^{-\frac{1}{2}}(\sqrt{t})e^{-\frac{r_d^2(x,y)}{(4+\epsilon)t}}.
\end{equation}
\end{theorem}

\begin{proof}[Proof of Theorem \ref{Thm:entropy bound area growth}]
On one hand, given $x\in \cN$, $r>0$, we have
\begin{equation}
\begin{split}
F_{x,r^2}(\Sigma)
=
\int_{\Sigma}\frac{H(x,y,r^2)}{r^{m-n}}d\mu(y)
\geq
&
C_1\int_{\Sigma}\frac{e^{-\frac{r_d^2(x,y)}{3r^2}}}{V_x(r)^{1/2}V_y(r)^{1/2}r^{m-n}}d\mu(y)\\
\geq
&
C\int_{\Sigma}\frac{e^{-\frac{r_d^2(x,y)}{3r^2}}}{r^{n}r^{m-n}}d\mu(y)\\
\geq
&
C\int_{\Sigma\cap B_r(x)}\frac{e^{-1/3}}{r^m}d\mu(y)\\
\geq 
&
C\frac{\area(\Sigma\cap B_r(x))}{r^m}.
\end{split}
\end{equation}
In the first inequality, we use the lower bound of the heat kernel; in the second inequality, we use Bishop-Gromov volume comparison inequality for a Riemannian manifold with non-negative Ricci curvature (See \cite[Section 7]{P}); in the third inequality, we restricted the integral on $\Sigma\cap B_r(x)$ and $r_d\leq r$ on this domain. By taking supremum of $x\in \cN$ and $r>0$ we get that the area growth is bounded from above by entropy.

\medskip
Next we show that the entropy is bounded from above by the area growth. Given $x\in\cN$, $r>0$. Let $r_{in}$ be the injective radius of $\cN$, $D$ be the diameter of $\cN$, and $nK>0$ be a Ricci curvature upper bound of $\cN$. Then by \cite[Proposition 14]{Cr}, for any $s\leq r_{in}/2$, $\vol(B_s(y))\geq  Cs^n$ for some dimensional constant $C$. Therefore, we have the heat kernel upper bound by taking $\epsilon=1/2$ (note $4+1/2\leq 5$):
\begin{equation}\label{Eq:Heat kernel upper bound injective radius}
\cH(x,y,r^2)\leq 
\begin{cases}
Cr^{-n} e^{-\frac{r_d^2(x,y)}{5r^2}}, \quad r < r_{in}/2;
\\
C r_{in}^{-n} e^{-\frac{r_d^2(x,y)}{5r^2}},\quad r\geq r_{in}/2.
\end{cases}
\end{equation}

We have two cases. 

\medskip
{\bf Case 1.} If $r\geq r_{in}/2$, then 
\begin{equation}
F_{x,r^2}(\Sigma)
\leq 
C\int_\Sigma \frac{e^{-\frac{r_d^2(x,y)}{5r^2}}}{r_{in}^{n} r^{m-n}}d\mu(y)
\leq 
C\area(\Sigma)D^{n-m}\leq C\kappa(\Sigma).
\end{equation}

\medskip
{\bf Case 2.} If $r < r_{in}/2$, let us pick a family of balls $\cB=\{B_r(x_1),\cdots,B_r(x_l)\}$ satisfies the following properties:
\begin{itemize}
\item $B_{r/2}(x_i)$ does not intersect $B_{r/2}(x_j)$ for $i\neq j$;
\item $\cup_{i=1}^lB_r(x_i)$ covers $\cN$.
\end{itemize}

For the sake of simplicity, we will use $B_i$ to denote the ball $B_{r/2}(x_i)$, and we will use $2B_i$ to denote the ball $B_r(x_i)$. Let us use $A_k$ to denote the annulus
\[
A_k=\{y\in \cN: kr\leq d_{x,y} \leq (k+1)r.\}
\]
By Bishop-Gromov inequality, $\vol(A_k)\leq Ck^{n-1}r^n$, where $n$ only depends on the dimension $n$. Note $2r<r_{in}$, So we can use \cite[Proposition 14]{Cr} again to show that $\vol{2B_i}\geq C r^{n}$. Therefore, $A_k$ is covered by at most $Ck^{n-1}$ number of $2B_i$. We will use $2B_j^k$ to denote the balls $\in\cB$ covering $A_k$. Then we can estimate
\begin{equation}
\begin{split}
F_{x,r^2}(\Sigma)
=&
\int_\Sigma \frac{\cH(x,y,r^2)}{r^{m-n}}d\mu(y)
\\
\leq
&
C \int_\Sigma \frac{e^{-\frac{r_d^2(x,y)}{5r^2}}}{r^n r^{m-n}}d\mu(y)
\\
\leq
&
C\sum_{k=0}^\infty \int_{\Sigma\cap A_k} \frac{e^{-\frac{r_d^2(x,y)}{5r^2}}}{r^m}d\mu(y).
\end{split}
\end{equation}
In the first inequality we use the upper bound of the heat kernel and note $r<r_{in}/2$, see (\ref{Eq:Heat kernel upper bound injective radius}).

A standard argument shows that
\[-r_d^2(x,y)\leq - (r_d(x,x_j)-r_d(y,x_j))^2\leq -\frac{1}{2}r_d^2(x,x_j)+r_d^2(y,x_j). \tag{$\star$}\]
Then we can estimate
\begin{equation}
\begin{split}
\int_{\Sigma\cap A_k} \frac{e^{-\frac{r_d^2(x,y)}{5r^2}}}{r^m}d\mu(y)
\leq &
Cr^{-m}\sum_{j}\int_{\Sigma\cap B_j^k}e^{-\frac{r_d^2(x,x_j)}{10r^2}}d\mu(y)\\
\leq &
C\kappa(\Sigma) \sum_j e^{-\frac{r_d^2(x,x_j)}{10r^2}}\\
\leq &
C\kappa(\Sigma) k^{n-1}e^{-Ck^2}.
\end{split}
\end{equation}
In the first inequality, we use ($\star$); in the third inequality we use the bound on the number of balls covering $A_k$, and note that $r_d(x,x_i)\geq k(r-1)$, and we use ($\star$) again.

Thus, we have
\[F_{x,r^2}(\Sigma)\leq C\kappa(\Sigma)\sum_{k=0}^\infty k^{n-1}e^{-Ck^2}\leq C\kappa(\Sigma)\]
when $r< r_{in}/2$. Combining two cases and taking supremum among $x\in\cN$, $r>0$ gives
\[\lambda(\Sigma)\leq C \kappa(\Sigma).\]
\end{proof}

Together with the monotonicity formula, Theorem \ref{Thm:entropy bound area growth} gives the following area growth bound of mean curvature flows.

\begin{corollary}\label{Cor:area growth bound of MCF}
Suppose $\cN$ is a closed Riemannian manifold with non-negative sectional curvatures and parallel Ricci curvatures. Let $\{M_t\}_{t\in[0,T)}$ be a $m$-dimensional mean curvature flow in $\cN$. Then there is a constant $C$ only depending on the initial entropy such that
\begin{equation}
\kappa(M_t)\leq C
\end{equation}
for any $t\in[0,T)$.
\end{corollary}

\begin{proof}
By Theorem \ref{Thm:entropy monotonicity} and Theorem \ref{Thm:entropy bound area growth}.
\end{proof}

For a general closed Riemannian manifold, we can still get an area growth bound, but it depends on the time interval on which the mean curvature flow exists.

\begin{corollary}\label{Cor:area growth bound of MCF weak version}
Suppose $\cN$ is a closed Riemannian manifold. Let $\{M_t\}_{t\in[0,T)}$ be a $m$-dimensional mean curvature flow in $\cN$. Then there is a constant $C$ only depending on the initial entropy and $T$ such that
\begin{equation}
\kappa(M_t)\leq C
\end{equation}
for any $t\in[0,T)$.
\end{corollary}

\begin{remark}
In order to apply the monotonicity formula of entropy to get the area growth bound, we require the ambient Riemannian manifold having non-negative sectional curvatures, hence we have proved Theorem \ref{Thm:entropy bound area growth} only for closed Riemannian manifolds with non-negative Ricci curvature. 

Nevertheless, Theorem \ref{Thm:entropy bound area growth} should hold true for much larger class of Riemannian manifolds. In fact, we only use the heat kernel bound (Theorem \ref{Thm:Heat kernel bound}) to get the area growth bound. There are much research on the estimate of heat kernel bound, and some of them may be used to improve Theorem \ref{Thm:Heat kernel bound}. This is out of the scope of this paper, and we only refer the readers to some nice surveys on this topic, see \cite{G}, \cite{S}. 
\end{remark}

\begin{remark}
It is hard to imagine that Corollary \ref{Cor:area growth bound of MCF} can only hold for Riemannian manifolds with such a curvature constraint. We conjecture that Corollary \ref{Cor:area growth bound of MCF} is true without any curvature assumption.
\end{remark}

Let us conclude this section by showing that the entropy of a closed embedded submanifold in a closed Riemannian manifold $\cN$ with non-negative Ricci curvature is finite.

\begin{proposition}\label{Prop:entropy is finite}
Suppose $\cN$ is a closed Riemannian manifold with non-negative Ricci curvature. Let $\Sigma$ be a closed embedded submanifold in $\cN$. Then $\lambda(\Sigma)<\infty$.
\end{proposition}

\begin{proof}
By Theorem \ref{Thm:entropy bound area growth}, we only need to show the area growth of $\Sigma$ is finite. Since $\Sigma$ is compact, it has bounded mean curvature. Moreover, if $\cN$ is isometrically embedded into $\R^N$, the mean curvature of $\Sigma$ in $\R^N$ is also bounded. Then a monotonicity of the area growth and a comparison of intrinsic and extrinsic balls gives the finiteness of the area growth bound of $\Sigma$ (see \cite[Section 17]{Si}).
\end{proof}

\section{Partial Regularity}\label{S:Partial Regularity}
Throughout this section, we will concentrate on the case that the submanifolds are closed embedded surfaces in a closed three Riemannian manifold $\cN$. We follow the idea of Ilmanen in \cite{I}. Let us first state an approximate graphical decomposition Lemma by Simon in \cite{Si2}.

\begin{theorem}[\cite{Si2}, Lemma 2.1; also see \cite{I}, Theorem 10]\label{Thm:Simon Lemma}
For $n\geq 3$, $D>0$, there exists $\epsilon_1=\epsilon_1(n,D)>0$ such that if $\epsilon<\epsilon_1$ and if $M$ is a smoothly embedded closed $2$-dimensional manifold in $\R^n$ such that
\begin{equation}
\int_{M\cap B_R}|A|^2 d\mu\leq \epsilon^2,\quad \area(M\cap B_R)\leq D\pi R^2,
\end{equation}
then there are pairwise disjoint closed disks $P_1,\cdots,P_N$ in $M\cap B_R$ such that
\begin{equation}\label{Eq:SiLemP}
\sum_j\diam(P_j)\leq C(n,D)\epsilon^{1/2}R,
\end{equation}
and for any $S\in[R/4,R/2]$ such that $M$ is transverse to $\partial B_S$ and $\partial B_S \cap(\cup_m P_m)=\emptyset$, we have
\[M\cap B_S=\cup_{l=1}^m D_l\]
where each $D_l$ is an embedded disk. Furthermore, for each $D_l$ there is a $2$-plane $L_l\subset\R^n$, a simply connected domain $\Omega_l\subset L_l$, disjoint closed balls $\overline{B_{l,p}}\subset\Omega_l$, $p=1,\cdots,p_l$ and a function 
\[u_l:\Omega_l\backslash \cup \overline{ B_{l,p}}\rightarrow L_l^\bot\]
such that
\begin{equation}\label{Eq:SiLem1}
\sup\left|\frac{u_l}{R}\right|+|D u_l|\leq C(n,D)\epsilon^{1/2(2n-3)}
\end{equation}
and
\begin{equation}\label{Eq:SiLem2}
D_l\backslash\cup_m P_m=\Graph(u_l|_{\Omega_l\backslash\cup_p\overline{B_{l,p}}}).
\end{equation}
\end{theorem}

\bigskip
Intuitively, this theorem illustrates the following picture: imagine we have a surface in $\R^n$, intersecting a ball $B_R$, with small total curvature and area growth bound. Then in a slightly smaller ball $B_S$, it is the union of embedded disks $D_i$. Moreover, besides some bad regions $P_j$ (which are called ``pimples" by Simon), $D_i$'s are very nice $C^{1,\alpha}$ graphs. In addition, the bad regions are still topological disks, and their diameters are bounded.

Theorem \ref{Thm:Simon Lemma} is stated for surfaces in a Euclidean space rather than a closed Riemannian manifold. In order to use this theorem, we first isometrically embed $\cN$ into an Euclidean space $\R^N$. From now on we will fix such an isometric embedding, and we will use $B_r^N(x)$ to denote the ball with radius $r$ in $\R^N$. The following lemma allows us to compare the area growth bound of a surface in $\cN$ with the area growth bound in $\R^N$.

\begin{lemma}\label{Lem:intrinsic diameter and extrinsic diameter}
There exists $\tau>0$ such that for any $x\in\cN$ and $r<\tau$, $(B^N_r(x)\cap \cN)\subset B_{2r}(x)$.
\end{lemma}

\begin{proof}
Since $\cN$ is compact, its second fundamental forms are uniformly bounded in $\R^N$. Then there exists $\tau>0$ and $\alpha>0$, $C>0$ such that at each point $x\in\cN$, $\cN$ is a smooth graph over its tangent space $T_x\cN\subset\R^N$ in a ball of radius $\tau$, and the $C^{1,\alpha}$ norm of the graph is bounded by $C$. Then by suitably choose smaller $\tau$, we will see that such a $\tau$ satisfies the statement in the lemma. 
\end{proof}

Using Lemma \ref{Lem:intrinsic diameter and extrinsic diameter}, we obtain an area bound of the disks in Simon's Theorem \ref{Thm:Simon Lemma}.

\begin{corollary}[\cite{I}, Corollary 11]\label{Cor:Area bound of disks}
Let $M\subset\cN\subset\R^N$ be a closed embedded surface. Assume the area growth $\kappa(M)\leq D$ for a constant $D$. Then under the hypotheses in Theorem \ref{Thm:Simon Lemma}, for any $x\in M\cap B_R$ such that $B_{2r}(x)\subset B_R$, the connected component $M'$ of $M$ of $M\cap B_r(x)$ containing $x$ is embedded and satisfies
\[\pi r^2(1- C\epsilon^{1/(2n-3)})\leq \area(M)\leq \pi r^2(1+C\epsilon^{1/2(2n-3)}).\]
Here $C$ is a constant only depending on $N$ and $D$.
\end{corollary}

\begin{proof}
The lower bound comes from (\ref{Eq:SiLem1}) and (\ref{Eq:SiLem2}), and 
\[\area(M')\geq \area(\Graph(u_l|_{\Omega'})).\]
The upper bound comes from (\ref{Eq:SiLem1}) and (\ref{Eq:SiLem2}) again, and
\[\area(M')\leq \area(\Graph(u_l|_{\Omega'}))+\sum_m\area(P_m).\]
Moreover, (\ref{Eq:SiLemP}) implies that $\sum_m\diam(P_m)$ has an upper bound as diameters in $\R^N$, and Lemma \ref{Lem:intrinsic diameter and extrinsic diameter} implies that $\sum_m\diam(P_m)$ has an upper bound as diameters in $\cN$. Then the above equations together with the area growth bound give the desired upper bound of $\area(M')$.
\end{proof}

Next we recall Allard's regularity theorem.

\begin{theorem}[\cite{A}, Section 8; also see \cite{I}, Theorem 9]\label{Thm:Allard regularity}
There exists $\epsilon_2=\epsilon_2(n,k)$ and $\delta=\delta(n,k)$ with the following significance. Suppose $\mu$ is an integer $k$-rectifiable Radon measure such that $|H|\in L_{\loc}^1(\mu)$. If $\epsilon<\epsilon_2$, $0\in\spt$ and $r>0$ such that
\[|H|\leq \frac{\epsilon}{r},\quad \text{for $\mu$-a.e. $x\in B_r$},\]
and 
\[\mu(B_r)\leq (1+\epsilon)\omega_k r^k,\]
then there is a $k$-plane $T$ containing $0$ and a domain $\Omega\subset T$, and a $C^{1,\alpha}$ vector valued function $u:\Omega\to T^\bot$ such that
\[\spt \mu\cap B_{\delta r}=\Graph u\cap B_{\delta r},\]
and
\[\sup\left|\frac{u}{r}\right|+\sup|Du|+r^\alpha[Du]_\alpha\leq C\epsilon^{1/4n}.\]
Here $[Du]_\alpha=\sup\frac{|Du(x)-Du(y)|}{|x-y|^\alpha}$.
\end{theorem}

The following theorem is due to Ilmanen \cite[Section 4]{I}. The similar idea has been used by Choi-Schoen in \cite{CS} to obtain the partial regularity of the limit of a sequence of minimal surfaces in a three manifold, cf. \cite{W}, \cite{CM2}. However, in \cite{CS}, \cite{W} and \cite{CM2}, all the surfaces in the sequence satisfy an equation, hence they all have nice locally graphical bounds. In Ilmanen's theorem, there is no such a graphical bound, so one needs \cite{Si2} to obtain a layer decomposition and an almost graphical bound on the limit. We sketch the proof for the sake of convenience of the readers.

\begin{theorem}\label{Thm:Partial regularity with H and area growth bound}
Let $\cN$ be a closed Riemannian manifold. Suppose $M_i\subset\cN$ is a family of closed smoothly embedded surfaces with uniformly bounded genus, and there is a constant $D$ such that 
\[\kappa(\Sigma)\leq D,\quad \int_{M_i}|H|^2d\mu_{M_i}\to 0.\]
Then there exists a subsequence (still denoted by $M_i$) converging to a varifold $V$ in the sense of varifold, and $\spt V$ is a smoothly embedded minimal surface.
\end{theorem}

\begin{proof}
Let us fix an isometric embedding of $\cN$ into a Euclidean space $\R^N$, and we will use $A^N$ and $H^N$ to denote the mean curvature and the second fundamental forms of a surface in $\R^N$ respectively. Since the genus of $M_i$'s are uniformly bounded, Gauss-Bonnet theorem shows that $\int_{M_{i}}|A|^2 d\mu_{M_{i}}$ is uniformly bounded by a constant $C$. Gauss-Codazzi identity shows that $\int_{M_{i}}|A^N|^2 d\mu_{M_{i}}$ is also uniformly bounded by a constant $C$.

Define $\sigma_j=|A^N|^2\bH^2\lfloor M_i$, which is a Radon measure. Here $\bH^2$ is the $2$-dimensional Hausdorff measure in $\R^N$. Then $\sigma_j$ has uniformly bound $\sigma_j(\R^N)\leq C$, and a subsequence of $\sigma_j$ (still denoted by $\sigma_j$) converges to a limit $\sigma$ by the compactness of Radon measure. Moreover $\sigma(\R^N)\leq C$. Define $\epsilon_0=\min\{\epsilon_1,\epsilon_2\}$ in Theorem \ref{Thm:Simon Lemma} and Theorem \ref{Thm:Allard regularity}. Then we define the concentration set 
\[
\cS=\{x\in\R^N:\sigma(x)>\epsilon_0\}.
\]
Since $\sigma(\R^N)\leq C$, $\cS$ is a finite set. For any $p\not\in\cS$, there exists $\epsilon<\epsilon_0$ and $r$ sufficiently small so that
\begin{equation}
\sigma(B_r(p))<\epsilon^2.
\end{equation}
Then for $k$ sufficiently large, $\sigma_j(B_r(p))<\epsilon^2$. Then by Theorem \ref{Thm:Simon Lemma}, there are $r_k\in[r/4,r/2]$ and a decomposition of $M_k\cap B_{r_k}(p)$ into the union of disks $D_{k,l}$. Each $D_{k,l}$ is an embedded disk, and by Corollary \ref{Cor:Area bound of disks} it satisfies the area growth bound
\[
\area(D_{k,l}\cap B_{\rho}(x))\leq (1+C(n,D)\epsilon^\gamma)\rho^2,\quad x\in D_{k,l},~ |x-p|+\rho\leq r_k.
\]

Then passing to a further subsequence (still denoted by the subscription $k$) we may assume that the limit of the number of the disks are bounded by $l_0$, and $\bH^2\lfloor D_{k,l}$ converging to a limit $\nu_l$ as varifolds for $l=1,\cdots,l_0$ by the compactness of varifolds. Moreover, by the mean curvature bounds of $M_k$, $\nu_l$ weakly solves the minimal surface equation in $\cN$. Then $\nu_l$ satisfies the assumptions in Allard's regularity Theorem \ref{Thm:Allard regularity}, and we can apply Allard's regularity Theorem to show that $\nu_l\cap B_{\sigma r_\infty}(p)$ is a graph of $C^{1,\alpha}$ function $u$ defined over a domain in a $2$-plane. Furthermore, since  $\nu_l$ weakly solves the minimal surface equation, Schauder estimates shows that $u$ is actually smooth.

In conclusion, we have proved that the limit varifold $V$ of $M_k$ (after passing to a subsequence) is supported on a union of smoothly closed embedded disks in $B_{\sigma r_\infty}(p)$. Then the maximum principle of minimal surfaces shows that $V$ is supported on a closed embedded minimal surface $\Sigma$, besides $\cS$. Then a removable of isolated singularities theorem by Gulliver \cite{Gu} (also see \cite[Proposition 1]{CS}) shows that $V$ is supported on a closed embedded minimal surface $\Sigma$ in $\cN$.
\end{proof}

\section{Applications to Mean Curvature Flow}\label{S:Applications to Mean Curvature Flow}
In this section, we apply the partial regularity theorem in the previous section to mean curvature flow. Recall that a mean curvature flow is ancient if it is defined for $t\in(-\infty,0)$, and a mean curvature flow is long-time if it is defined for $t\in[0,\infty)$. The simplest examples of ancient and long-time mean curvature flows are minimal submanifolds, which are static solutions to mean curvature flow.

Ancient mean curvature flow and long-time mean curvature flow always has a weak limit as $t\to-\infty$ or $t\to \infty$ respectively.

\begin{theorem}
Suppose $M_t$ is an ancient mean curvature flow in a closed Riemannian manifold $\cN$. If $\lim_{t\to-\infty}\area(M_t)<\infty$, then there exists a sequence $t_i\to-\infty$ and a stationary varifold $V$, such that $M_{t_i}\to V$ in the sense of varifolds.

Suppose $M_t$ is a long-time mean curvature flow in a closed Riemannian manifold $\cN$. Then there exists a sequence $t_i\to\infty$ and a stationary varifold $V$, such that $M_{t_i}\to V$ in the sense of varifolds.
\end{theorem}

\begin{proof}
By the first variational formula formula, the derivative of the area of $M_t$ satisfies the following identity
\[\partial_t \area(M_t)=-\int_{M_t}|H|^2 d\mu_{M_t}.\]
This identity together with $\lim_{t\to-\infty}\area(M_t)< \infty$ implies that $\area(M_t)< \infty$. Hence we could pick a sequence of $t_i$, converging to $\pm\infty$ if $M_t$ is long-time or ancient respectively, such that $\int_{M_{t_i}}|H|^2d\mu_{M_{t_i}}\to 0$. Then the compactness of varifolds implies that $M_{t_i}$ has a varifold limit $V$, which is stationary according to the mean curvature bound.
\end{proof}

In general the regularity of the limit $V$ is not known. Now we can apply the partial regularity Theorem \ref{Thm:Partial regularity with H and area growth bound} to obtain a partial regularity of the limit of mean curvature flow of surfaces in special $3$-dimensional Riemannian manifolds, which is our Theorem \ref{thm:Main theorem 2}.

\begin{proof}[Proof of Theorem \ref{thm:Main theorem 2}]
Suppose $\lambda(M_0)=\lambda_0$. Proposition \ref{Prop:entropy is finite} implies that $\lambda_0$ is finite. Monotonicity formula of entropy implies that $\lambda(M_t)\leq\lambda_0$, hence Theorem \ref{Thm:entropy bound area growth} implies that the area growth of $M_t$ is uniformly bounded (cf. Corollary \ref{Cor:area growth bound of MCF}). Monotonicity of area of $M_t$ implies that we could a sequence of $t_i$ such that $\int_{M_{t_i}}|H|^2d\mu_{M_{t_i}}\to 0$. Then Theorem \ref{Thm:Partial regularity with H and area growth bound} implies that there is a subsequence of $t_i$ (still denoted by $t_i$) such that $M_{t_i}$ converges to a varifold $V$ in the sense of varifolds, and $\spt V$ is a smoothly embedded minimal surface.
\end{proof}

Almost exactly the same proof shows Theorem \ref{thm:Main theorem 3}. We omit the proof here. 

\section{Some Questions}
Let us conclude this paper by asking some questions. 

\begin{question}
Can we obtain an area growth bound for long-time mean curvature flows in a closed Riemannian manifold without the curvature assumptions?
\end{question}

If this is true, then we can still obtain the partial regularity of the long-time limit of mean curvature flow of surfaces. 

The rescaled mean curvature flow (see \cite{Hu}, \cite[Section 2]{CIMW}) is the gradient flow of Gaussian area in $\R^n$, which has finite area growth bound as $t\to\infty$. The Gaussian measured space is far from the assumptions in Theorem \ref{Thm:Hamilton's monotonicity}. Although the rescaled mean curvature flow is not a mean curvature flow, it shares some similarities to mean curvature flow. Therefore we conjecture that there might be some weaker assumptions to ensure that a mean curvature flow has long-time area growth bound. Otherwise, there exists a closed Riemannian manifold such that the long-time mean curvature flow would become ``complicated" as time goes to infinity. 
\bigskip

\begin{question}
Does the equivalence of the entropy and area growth still hold in other ambient Riemannian manifolds?
\end{question}

It would also be interesting to know that whether the results in Section \ref{S:Area Growth} are still valid for more general Riemannian manifolds. It seems plausible because there are many studies on heat kernel and there might be some useful properties of heat kernel which are out of our scope.

\bigskip
\begin{question}
What is the multiplicity of the convergence of $M_{t_i}$ to $V$ in Theorem \ref{thm:Main theorem 2} and Theorem \ref{thm:Main theorem 3}?
\end{question}

If the convergence has multiplicity $1$, Brakke's regularity theorem implies that the convergence of $M_{t_i}$ to $V$ is actually smooth. Moreover, with the multiplicity $1$ assumtion, the tangent flows of mean curvature flow are unique in many cases, see \cite{Sch}, \cite{CM3} and \cite{CSch}. We guess that the uniqueness of the limit of a long-time mean curvature flow may be true as well.

This conjecture is inspired by the famous multiplicity $1$ conjecture by Ilmanen in \cite{I}. In the rescaled mean curvature flow case, Ilmanen conjecture that the multiplicity $1$ of the tangent flow is always $1$. This conjecture is still open.

\end{document}